\documentclass[12pt,draftcls,onecolumn]{IEEEtran}
\usepackage{fullpage}
\usepackage{latexsym,url}
\usepackage{amsmath,amssymb,array,graphicx,palatino,verbatim}
\usepackage{setspace}
\usepackage
  [breaklinks,bookmarks,bookmarksnumbered,bookmarksopen,bookmarksopenlevel=2]
  {hyperref}
{\makeatletter
 \gdef\xxxmark{%
   \expandafter\ifx\csname @mpargs\endcsname\relax % in minipage?
     \expandafter\ifx\csname @captype\endcsname\relax % in figure/caption?
       \marginpar{xxx}% not in a caption or minipage, can use marginpar
     \else
       xxx % notice trailing space
     \fi
   \else
     xxx % notice trailing space
   \fi}
 \gdef\xxx{\@ifnextchar[\xxx@lab\xxx@nolab}
 \long\gdef\xxx@lab[#1]#2{{\bf [\xxxmark #2 ---{\sc #1}]}}
 \long\gdef\xxx@nolab#1{{\bf [\xxxmark #1]}}
}

\let\realbfseries=\bfseries
\def\bfseries{\realbfseries\boldmath}

\newif\ifabstract
\abstracttrue
\newif\iffull
\ifabstract \fullfalse \else \fulltrue \fi

\let\epsilon=\varepsilon

\newsavebox{\theorembox}
\newsavebox{\factbox}
\newsavebox{\lemmabox}
\newsavebox{\remarkbox}
\newsavebox{\corollarybox}
\newsavebox{\propositionbox}
\newsavebox{\examplebox}
\newsavebox{\conjecturebox}
\newsavebox{\algbox}
\newsavebox{\qbox}
\newsavebox{\problembox}
\newsavebox{\definitionbox}
\newsavebox{\assumptionbox}
\newsavebox{\hypothesisbox}
\savebox{\theorembox}{\noindent\bf Theorem}
\savebox{\factbox}{\noindent\bf Fact}
\savebox{\lemmabox}{\noindent\bf Lemma}
\savebox{\remarkbox}{\noindent\bf Remark}
\savebox{\corollarybox}{\noindent\bf Corollary}
\savebox{\propositionbox}{\noindent\bf Proposition}
\savebox{\examplebox}{\noindent\bf Example}
\savebox{\conjecturebox}{\noindent\bf Conjecture}
\savebox{\algbox}{\noindent\bf Algorithm}
\savebox{\definitionbox}{\noindent\bf Definition}
\savebox{\problembox}{\noindent\bf Problem}
\savebox{\assumptionbox}{\noindent\bf Assumption}
\savebox{\hypothesisbox}{\noindent\bf Hypothesis}
\newtheorem{theorem}{\usebox{\theorembox}}
\newtheorem{lemma}[theorem]{\usebox{\lemmabox}}

\newtheorem{definition}{\usebox{\definitionbox}}

\newcommand{\qed}{\;\;\;\Box}

\newenvironment{proof}{\par{\bf Proof:}}{\(\qed\) \par}

\begin{document} 

\title{\LARGE{Finding an Integral vector in an Unknown Polyhedral Cone}}
%\author{\authorblockN{Ali Kakhbod}
%\authorblockA{ \\
%University of Michigan, Ann Arbor\\
%Email: {akakhbod@umich.edu}\\
%\and
%\authorblockN{Morteza Zadimoghaddam}
%\authorblockA{ \\
%MIT, Cambridge\\
%Email: morteza@mit.edu}
%}}
\author{{\large{Ali~Kakhbod and Morteza Zadimoghaddam} \\
\normalsize{University of Michigan, Massachusetts Institute of Technology}\\
\normalsize{Email:  \tt{akakhbod@umich.edu},  \tt{morteza@mit.edu}}}}

%\footnote{Department of Electrical Engineering and Computer Science, University of Michigan, Ann Arbor.  Email: \tt{akakhbod@umich.edu}} and Morteza Zadimoghaddam\footnote{Department of  Computer Science, MIT, Cambridge. Email: \tt{morteza@mit.edu}}}} 

%\author{\authorblockN{Ali Kakhbod}
%\authorblockA{ \\
%University of Michigan, Ann Arbor\\
%Email: {akakhbod@umich.edu}
%\and
%\authorblockN{Morteza Zadimoghaddam}
%\authorblockA{ \\
%MIT, Cambridge\\
%Email: morteza@mit.edu}
%}}

\maketitle

\begin{abstract}
 We present an algorithm to find an integral vector  in the polyhedral cone  $\Gamma=\{X | \textbf{A}X \leq \textbf{0}\}$, without assuming the explicit knowledge of $\textbf{A}$. About the polyhedral cone, $\Gamma$,  it is  only given that, (i) the elements of \textbf{A} are in 
$\{-d,-d+1,\cdots,0,\cdots,d-1,d\}$, $d \in \mathbb{N}$,  and, (ii) $Y=[y(1),y(2),\cdots,y(n)]$ is a non-zero integral  solution to $\Gamma$. The proposed algorithm finds a non-zero integral vector in $\Gamma$ such that its maximum element is less than ${(2d)^{2^{n-1}-1}}/{2^{n-1}}$.
%In this paper we consider linear system of inequalities, $\textit{A}x \leq \textbf{b}$, $\textit{A} \in \mathbb{N}^{m\times n}$ and $\textbf{b} \in \mathbb{N}^m$, for the case that matrix $\textit{A}$ and vector $\textbf{b}$ are completely unknown. We call it incomplete linear system of inequalities (\textit{ILSI}). We are interested in  deriving  an integer solution with minimal maximum entry based on the following given information about the linear system; (i) elements of matrix $\textit{A}$ are in $\{-d,-d+1,\cdots, -1, 0, 1, \cdots, d\}$ ($d$ is a positive integer number), (ii) elements of vector $\textbf{b}$ are non-negative integer numbers, (iii) a non-zero integer solution of the ILSI is given. The solution that we proposed is based on a constructive algorithm. 
\\

\end{abstract}
\begin{section}{Introduction}
 Finding an integral vector\footnote{An integral vector is a vector with non-negative integer elements.} in the polyhedral cone $\Gamma=\{X | \textbf{A}X \leq \textbf{0}\}$, for a given/known matrix $\textbf{A}$, $\textbf{A} \in \mathbb{Z}^{n \times n}$, is a problem which has been considered in  great detail \cite{1,2,4,5}. 
 %So far,  many algorithms have been discovered to tackle this problem. For example, Fourier-Motzkin elimination method \cite{4,5}, which is based on eliminating variables, is one option that can be employed\footnote{Although the Fourier-Motzkin is a
%simple method but it is not effective for large systems.}  for finding an integral vector in the given polyhedral cone.  \\
 
%As it is mentioned in the above, for the case where $\textbf{A}$ is known, finding an integral vector in $\Gamma$ (if there exists any) might be possible  via Fourier-Motzkin elimination approach.

In this paper, we consider the above-mentioned problem from another angle with a distinctly different assumption. Here, we  assume  there is no explicit knowledge about $\textbf{A}$ (i.e., \textbf{A} is unknown), but a non-zero integer solution of  $\Gamma$ is given. Under these assumptions, we can show that 
not only does there exist another integer solution of $\Gamma$ but also the maximum element of the obtained solution is less than $\frac{(2d)^{2^{n-1}-1}}{2^{n-1}},$ when elements of $\textbf{A}$ are in $\{-d,-d+1,\cdots,0,\cdots,d-1,d\}$\footnote{In other words, suppose that $Y=[y(1),y(2),\ldots,y(n)]$ is a given non-zero integer solution of $\Gamma$ such that $\max_{1 \leq i \leq n}y(i)>\frac{(2d)^{2^{n-1}-1}}{2^{n-1}}$. Then, without knowing $\textbf{A}$,  we can find another non-zero integer solution of $\Gamma$ such that its maximum element is less than $\frac{(2d)^{2^{n-1}-1}}{2^{n-1}}$.}, where $d \in \mathbb{N}$.
\\

The rest of the paper is organized as follows. In Section \ref{sec1} we present the main theorem of the paper. In Section \ref{sec2} we consider an example. In Section \ref{sec3} we present some applications of the algorithm.

 %\textit{Is it doable to find a non-zero integral vector in the polyhedral cone $\Gamma=\{X | \textbf{A}X \leq \textbf{0}\}$ when $\textbf{A}$ is not explicitly known?}\\
 
%This question may arise  in many fields related to uncertainty, for instance:  chaotic systems, communication systems and cryptography.\\

\end{section}

%%%%%%%%%%%%%%%%%%%%%%%%%%%%%%%%%%%%%%%%%%%%%%%%%%%%%%%%%%%%%%%%%%%%%%%%%%%%%%%%%%%%%%%%%%%%%%%%%%%%%%%%%%%%%%%%%%%%%%%%%%%%%%%%%555555

\begin{section}{Main Theorem}
\label{sec1}

In this section, we present the main theorem of this paper.  The proof of this theorem is constructive and   along the proof  we present the algorithm that archives the desired properties. 
%to achieve another  solution (integral vector) for the polyhedral cone. Furthermore,  we show that the maximum component of the obtained integer solution is bounded above by $\frac{(2d)^{2^{n-1}-1}}{2^{n-1}}$. \\

\begin{theorem}
\label{1}
Consider $\Gamma=\{X| \textbf{A}X\leq\textbf{0}\}$, with this knowledge that the elements of  $\textbf{A}$ are in $\{-d,-d+1,\cdots,0,\cdots,d-1,d\}$.  Assume that $Y=[y(1),y(2),\ldots,y(n)]$ is a given (arbitrary) non-zero integer solution of $\Gamma$. Then  there exists an integer solution $X=[x(1),x(2),\ldots,x(n)]$ which satisfies $\textbf{A}X\leq \textbf{0}$  and 
\begin{eqnarray}
\label{bound1}
\max_{1\leq i \leq n} x(i) \leq \frac{(2d)^{2^{n-1}-1}}{2^{n-1}}.
\end{eqnarray}
\end{theorem}

\begin{proof}
 Without loss of generality we assume that,  
$$y(1) \leq y(2) \leq \ldots \leq y(n).$$ 
%By considering all the possible equations, we see that the constraints satisfied by any solution  $X$ 
%which can be derived by $Y$ define a polyhedral cone  $\Lambda_Y^*$ \cite{6}.
%Observe that each constraint defining $\Lambda_Y^*$ is of the form
%\begin{displaymath}
%\sum_{i=1}^{n} c_i x(i) \leq 0
%\end{displaymath}
%where $c_i \in \{-d, \cdots,d\}$ for each $i$ and $\sum_{i=1}^{n} c_i y(i)\leq 0.$
We  consider the polyhedral cone $\Lambda_Y$, $\Lambda_Y \subset \Gamma$, which is defined by {\em all}
inequalities of the form 
\begin{displaymath}
\sum_{i=1}^{n} c_i x(i) \leq 0,
\end{displaymath}
where $c_i \in \{-d, \cdots,d\}$ for all $i$ and $\sum_{i=1}^{n} c_i y(i) \leq 0$.

%Since every constraint defining $\Lambda_Y^*$ also defines $\Lambda_Y$,
%any solution $X \in \Lambda_Y$ satisfies $X \in \Lambda_Y^*$, In words, $\Lambda_{Y}^*$ is a polyhedron that can be resulted by $Y$ when $\textbf{A}$ is completely known, but in our problem $\textbf{A}$ is not explicitly known, thus we consider $\Lambda_{Y}$ which is resulted from all the inequalities with coefficients in $\{-d,\cdots,d\}$ ($c_i \in \{-d, \cdots,d\}$ for all $i$) such that $\sum_{i=1}^{n} c_i y(i) \leq 0$, therefore, $\Lambda_Y \subseteq \Lambda_Y^*$, and if $X \in \Lambda_Y$ then $X \in \Lambda_Y^*$.

We will describe a procedure to construct a solution  
$X \in \Lambda_Y$ that satisfies $$\max_{1 \leq i \leq n} x(i) \leq 
\frac{(2d)^{2^{n-1}-1}}{2^{n-1}}.$$  

\begin{definition}
\label{rho}
The sequence $\Upsilon_j$ is defined by the recurrence relation 
$\Upsilon_1=d, \; \Upsilon_j=2\Upsilon_{j-1}^2$.
\end{definition}
  Thus
\begin{displaymath}
\Upsilon_j= \frac{1}{2}(2d)^{2^{j-1}}. 
\end{displaymath}
For $j \in \{1, \ \dots , \ n-1 \}$, let $\Lambda_Y^j$ denote the 
polyhedral cone in $n+1-j$ dimensions which is defined by {\em all}
constraints of the form
\begin{displaymath}
\sum_{i=j}^{n} c_i x(i) \leq 0,
\end{displaymath}
where for all $i, \; c_i$ is an integer with $|c_i | \leq \Upsilon_j$,
and $\sum_{i=j}^{n} c_i y(i) \leq 0.$

Our procedure begins by setting $x(n)=1$.
For integers $j$ decreasing from $n-1$ down to $1$, we describe a way to
select and update the partial solutions, that is, in each iteration, say $n-j$, we derive a feasible solution with lower dimension called partial solution in that iteration and denote by $X^{(n-j)}$= $\{x(i): \; i \geq j \}
\in \Lambda_Y^{j}$.

For $j=n-1$ and $x(n)=1$, any {\em real} partial solutions 
$X^{(n-1)} \in \Lambda_Y^{n-1}$ is calculated by constraints of the form
\begin{equation}
\gamma_{n-1}^{u} x(n-1) \; \leq \; \gamma_{n}^u \label{eq:i1}
\end{equation}
or
\begin{equation}
\gamma_{n-1}^{l} x(n-1) \; \geq \; \gamma_{n}^l, \label{eq:i2}
\end{equation}
where 
$\gamma_{n-1}^{u}, \; \gamma_{n}^u, \; \gamma_{n-1}^{l},$ and
$\gamma_{n}^l$ are integers satisfying
\begin{eqnarray}
0 & \leq & \gamma_{n-1}^{u} \; \leq \; \Upsilon_{n-1} \\
1 & \leq & \gamma_{n}^{u} \; \leq \; \Upsilon_{n-1} \\
1 & \leq & \gamma_{n-1}^{l} \; \leq \; \Upsilon_{n-1} \\
0 & \leq & \gamma_{n}^{l} \; \leq \; \Upsilon_{n-1} \\
\gamma_{n-1}^{u} y(n-1) & \leq & \gamma_{n}^u y(n) \\
\gamma_{n-1}^{l} y(n-1) & \geq & \gamma_{n}^l y(n) 
\end{eqnarray}
We initially choose $x(n-1)$ to be a positive number that satisfies
all constraints in $\Lambda_Y^{^{n-1}}$ assuming that $x(n)=1$. 
More specifically, we choose positive integers 
$\gamma_{n-1}^{*} \leq \Upsilon_{n-1}$ and $\gamma_{n}^{*} \leq \Upsilon_{n-1}$ 
such that
\begin{displaymath}
\frac{\gamma_{n}^{*}}{\gamma_{n-1}^{*}} \; = \; x(n-1) \; \leq \; x(n)
\; = \; 1,
\end{displaymath}
and the constraints of $\Lambda_Y^{n-1}$ are satisfied.
We next multiply $x(n-1)$ and $x(n)$ by $\gamma_{n-1}^{*}$ to obtain
an {\em integral} partial solution satisfying
\begin{displaymath}
1 \; \leq x(n-1) \; \leq \; x(n) \; \leq \; \gamma_{n-1}^{*} 
\; \leq \; \Upsilon_{n-1} ,
\end{displaymath}
and this will be our initial partial solution when we begin to
consider $x(n-2)$.

For $j$ decreasing from $n-2$ down to $1$, suppose that we have an integral
partial solution $X^{(n-(j+1))^*}=\{x^*(i): \; i \geq j+1 \} \in \Lambda_Y^{j+1}$ with
$x^{*} (n) \leq \prod_{i=j+1}^{n-1} \Upsilon_i$.  We will use this partial
solution to construct an integral partial solution
$X^{(n-j)} \in \Lambda_Y^j$ with $x(n) \leq \prod_{i=j}^{n-1} \Upsilon_i$.  
We begin by setting $x(i)=x^{*} (i), \; j+1 \leq i \leq n$.
Assuming that this is a legitimate assignment, in order for {\em real}
partial solution $X^{(n-j)}$ to be an element of $\Lambda_Y^j$, we must have
that
\begin{displaymath}
\sum_{i=j}^{n} c_i x(i) \leq 0,
\end{displaymath}
where for all $i, \; c_i$ is an integer with $|c_i | \leq \Upsilon_j$,
and $\sum_{i=j}^{n} c_i y(i) \leq 0.$  
There are three cases to consider for $c_j$:
\begin{enumerate}
\item If $c_j = 0$, then since $X^{(n-(j+1))^*} \in \Lambda_Y^{j+1}$ it follows that
$\sum_{i=j+1}^{n} c_i x(i) \leq 0$ holds when $c_i$ is an integer with
$|c_i | \leq \Upsilon_j$ for all $i \geq j+1$, and when
$\sum_{i=j+1}^{n} c_i y(i) \leq 0$.
\item If $c_j = c_j^u > 0$, then we obtain an upper bound on $x(j)$:
\begin{equation}
x(j) \; \leq \; - \frac{1}{c_j^u} \sum_{i=j+1}^{n} c_i^u x^{*} (i).
\label{eq:u}
\end{equation}
Observe that
\begin{equation}
y(j) \; \leq \; - \frac{1}{c_j^u} \sum_{i=j+1}^{n} c_i^u y(i).
\label{eq:ug}
\end{equation}
\item If $c_j = c_j^l < 0$, then we obtain a lower bound on $x(j)$:
\begin{equation}
x(j) \; \geq \; - \frac{1}{c_j^l} \sum_{i=j+1}^{n} c_i^l x^{*} (i).
\label{eq:l}
\end{equation}
Observe that
\begin{equation}
y(j) \; \geq \; - \frac{1}{c_j^l} \sum_{i=j+1}^{n} c_i^l y(i).
\label{eq:lg}
\end{equation}
\end{enumerate}
In order for this approach to lead to a valid partial solution,
we need to guarantee that all upper bounds on $x(j)$ exceed all lower
bounds on $x(j)$.  By (\ref{eq:u})-(\ref{eq:lg}) we want to establish that
\begin{equation}
- \frac{1}{c_j^u} \sum_{i=j+1}^{n} c_i^u x^{*} (i) \; \geq \;
- \frac{1}{c_j^l} \sum_{i=j+1}^{n} c_i^l x^{*} (i) \label{eq:c}
\end{equation}
when
\begin{equation}
- \frac{1}{c_j^u} \sum_{i=j+1}^{n} c_i^u y(i) \; \geq \;
- \frac{1}{c_j^l} \sum_{i=j+1}^{n} c_i^l y(i). \label{eq:gc}
\end{equation}
Constraint (\ref{eq:c}) is equivalent to the condition
\begin{equation}
\sum_{i=j+1}^{n} (c_i^l c_j^u - c_i^u c_j^l ) x^{*} (i) \leq 0.
\label{eq:c2}
\end{equation}
The property (\ref{eq:gc}) can be rewritten
\begin{equation}
\sum_{i=j+1}^{n} (c_i^l c_j^u - c_i^u c_j^l ) y(i) \leq 0.
\label{eq:gc2}
\end{equation}
Notice that since $|c_i^l | \leq \Upsilon_j$ and $|c_i^u | \leq \Upsilon_j$ for all
$i \geq j$, it follows that 
$|c_i^l c_j^u - c_i^u c_j^l | \leq 2\Upsilon_j^2 = \Upsilon_{j+1}$ for all $i \geq j+1$.
Since (\ref{eq:gc2}) holds and $X^{(n-(j+1))*} \in \Lambda_Y^{j+1}$ by assumption,
it follows that (\ref{eq:c2}) holds.

We choose $x(j)$ to be the maximum value satisfying all constraints
(\ref{eq:u}) and (\ref{eq:l}).  Since $x(j)$ may be of the form
$\sigma / \gamma_j^{*}$ for integer $1 \leq \gamma_j^{*} \leq \Upsilon_j$ and
integer $\sigma$, we multiply the partial solution by
$\gamma_j^{*}$ to obtain an integral partial solution satisfying
\begin{displaymath}
x(n) \; \leq \; \gamma_j^{*} \prod_{i=j+1}^{n-1} \Upsilon_i \; \leq \;
\prod_{i=j}^{n-1} \Upsilon_i .
\end{displaymath}

Observe that $\Lambda_Y^1 = \Lambda_Y$, so at the end of the procedure
we have a solution vector $X$ with
\begin{displaymath}
0 \; \leq \; x(1) \; \leq \; \dots \; \leq \; x(n) \; \leq \;
\prod_{i=1}^{n-1} \Upsilon_i  \; = \; \prod_{i=1}^{n-1}\frac{1}{2}(2d)^{2^{i-1}}=\frac{(2d)^{2^{n-1}-1}}{2^{n-1}}.
\end{displaymath}
%that for all $1 \leq i \leq n$ satisfies $x(i) > 0$ if and only if $y(i) > 0$.
Notice that, since $0 \; \leq \; y(1) \; \leq \; \dots \; \leq \; y(n)$,   we also have $0 \; \leq \; x(1) \; \leq \; \dots \; \leq \; x(n)$, because of the definition of $\Lambda_Y$. 
\end{proof}

In the following   we present an example to illustrate the procedure.

%\subsubsection*{Example of a Construction of a Distance Function:}
\begin{section}{An Example}
\label{sec2}
Let $\Gamma=\{X|\textbf{A}X \leq \textbf{0}\}$,  where  \textbf{A} is a $4\times 4$ matrix and it is only given that its elements  are in $\{-1,0,1\}$. In addition, it is known that  $Y=[2,3,7,29]$ is an integral solution of $\Gamma$. Now, we find  an integral  vector in $\Gamma$ such that its maximum element is 8, which satisfies the proposed bound, namely,  it is less than $\frac{(2d)^{2^{n-1}-1}}{2^{n-1}}\Big|_{d=1,n=4}=16.$\\

 Recall that $\Upsilon_{1}=1, \Upsilon_{2}=2$ and $\Upsilon_{3}=8$.  We construct a solution  $[x(1),x(2),x(3),x(4)]$ for 
$\Lambda_Y$ when $y(1)=2, y(2)=3, y(3)=7$ and $y(4)=29$ .
 
\begin{enumerate}
\item Initialize $x(4)=1$.
\item $\Lambda_Y^{3}$ is the polyhedral cone with constraints 
$c_3 x(3)+ c_4 x(4)\leq 0$, where $c_3$ and $c_4$ are integers
with $|c_3| \leq 8$, $|c_4| \leq 8$, and $c_3 y(3)+c_4 y(4)=7c_3+29c_4 \leq 0$.
Observe that the defining inequalities for $\Lambda_Y^3$ for this example are
$0\leq x(3), 4x(3)\leq x(4)$ and $5x(3)\geq x(4)$; all the other inequalities
that we consider are less restrictive.
\item Initialize $x(3)=1/4$. %\item 
Multiply the solution by $4$ to obtain the integral partial solution $x(3)=1$ and $x(4)=4$.
\item $\Lambda_Y^{2}$ is the polyhedral cone with constraints 
$c_2 x(2) + c_3 x(3)+ c_4 x(4) \leq 0$, where $c_2, c_3$ and $c_4$ are integers
with $|c_2| \leq 2, |c_3| \leq 2, |c_4| \leq 2$, and 
$c_2 y(2) + c_3 y(3)+c_4 y(4)=3c_2+7c_3+29c_4 \leq 0$.
Observe that the defining inequalities for $\Lambda_Y^2$ for this example are
$0\leq x(2), 2x(2) \leq x(3)$ and $2x(2)+2x(3)\leq x(4).$
%\item Our new partial distance function is $f(2)=1, f(3)=2$ and $f(4)=7$.
\item Initialize $x(2)=1/2$. %\item 
Multiply the solution by $2$ to obtain the integral partial solution
 $x(2)=1$, $x(3)=2$ and $x(4)=8$.
\item $\Lambda_Y = \Lambda_Y^{1}$ is the polyhedral cone with constraints 
$c_1 x(1) +c_2 x(2) + c_3 x(3)+ c_4 x(4) \leq 0$, where $c_1, c_2, c_3, c_4
\in \{-1,0,1\}$, and
$c_1 y(1) + c_2 y(2) + c_3 y(3)+c_4 y(4)=2c_1+3c_2+7c_3+29c_4 \leq 0$.
Observe that the defining inequalities for $\Lambda_Y^1$ for this example are
$x(1)\geq 0, x(1)\leq x(2), x(1)+x(2)\leq x(3)$ and $x(1)+x(2)+x(3)\leq x(4)$.
\item Choose $x(1)=1, x(2)=1, x(3)=2$ and $x(4)=8$.
\end{enumerate}
The output of the algorithm is $X=[1,1,2,8]$. Thus, $X$ is another integral vector in $\Gamma$.
\end{section}
\begin{section}{Applications}
\label{sec3}
% This problem may arise  in many fields when we have uncertainty about the system. Specially, in situations where  an integer solution of $\Gamma$ is needed but $\textbf{A}$ is not known and the available integer solution to $\Gamma$ is useless because it is too big.....
%but specifically...
Some applications of the algorithm are as follows.
\begin{itemize}
\item An important feature of our algorithm is that we do not need to know matrix $\textbf{A}$ to produce a bounded solution for the linear program. In many situations, we have a solution that satisfies the requirements of a linear system, but we want another solution with some upper bounds on its size. It might be impossible (or very time consuming) to measure entries of matrix $\textbf{A}$. Using this method, we can have the bounded solution without knowing $\textbf{A}$. 

\item In some streaming applications, we might have the unbounded solution in advance, but matrix $\textbf{A}$ arrives later as a query, and we have to produce a bounded solution based on $\textbf{A}$. In streaming problems, we have very limited time for processing the query. In our approach, we can solve the problem without the knowledge of query in advance. When the query arrives, we already have the answer. 

\item Another interesting problem, that our bounds are useful for it, is the ellipsoid algorithm. Assume that we want to find a feasible solution for a linear program.  We have to start with an ellipsoid that contains at least one  feasible solution in the core. We then iteratively narrow down, and find a smaller ellipsoid. 
We have to start with an ellipsoid that contains some feasible non-zero integral solution of our linear program. In fact there has to be a non-zero integral point in the intersection of our polyhedral cone and the starting ellipsoid. If we start with a very small ellipsoid, we might not satisfy this property. If the starting ellipsoid is very large, we can not prove good bounds on the running time of the ellipsoid algorithm. 
In this case, we have to know some bounds on the size of the starting ellipsoid. Our algorithm can be used to get some bounds on the size of the starting ellipsoid. Our current bounds do not give polynomial bounds on the running time of ellipsoid algorithm, but we hope this approach can lead to such bounds by some polynomials. 

\item Another potential application of our algorithm is for finding the routing capacity regions of networks. It is  known that routing capacity regions of networks can be characterized using Farkas lemma as the solution set of infinite set of linear inequalities. But, our algorithm  gives  an upper bound (finite) on the set of inequalities needed to characterize the  capacity regions, \cite{serap}. 
\end{itemize}

\end{section}

\begin{comment}
\begin{subsection}{Remarks}
\begin{itemize}
	\item \textbf{remark.1}
Suppose we run the proposed algorithm $k$ times, and we derive\\
 $X^1=[x^1(1),x^1(2),\ldots,x^1(n)]^T$,\\
  $X^2=[x^2(1),x^2(2),\ldots,x^2(n)]^T$,\\
   $\vdots,$\\
    $X^k=[x^k(1),x^k(2),\ldots,x^k(n)]^T$,\\
     then every integer point (vector with integer elements) in the resulted convex hull of these points is another solution of the system of inequalities.

\item \textbf{remark.2}
Number of inequalities  after the $i$th step does not exceed $[\frac{1}{2}(2d)^{2^{i+1}-1}]^{(n-i+1)}$. As it is proved (in the proof of Theorem \ref{1}) After the $i$th step, the coefficients of the inequalities are in range $[-\Upsilon_i, \Upsilon_i]$. Therefore, the proposed algorithm runs in time $$O(n\times (2d)^{2^n}) \times poly(2^n)$$ because, the number of inequalities after the $i$th step is at most $[2\Upsilon_i+1]^{n-i}$, thus, the maximum number of inequalities is at most $O(\Upsilon_n)$ when $i$ is equal to $n$. On the other hand manipulating these inequalities can be done in time $poly(2^n)$ because the length of the coefficients does not exceed $poly(2^n)$.   
\end{itemize}
\end{subsection}
\end{comment}
\end{section}

\subsection*{Acknowledgments}
The authors are indebted to   Serap Savari  for stimulating and informative discussions to this work. They are also grateful to Katta Murty, Ramesh Saigal, and Demos Teneketzis 
whose comments significantly improved the presentation of the paper.

\end{document}